\documentclass{amsart}

\usepackage{lmodern} 
\usepackage{microtype} 
\usepackage[english]{babel} 
\usepackage{hyperref} 
\usepackage{mathrsfs}

\theoremstyle{plain} 
\newtheorem{lemma}{Lemma} 
\newtheorem*{theoremA}{Theorem A} 
\newtheorem*{theoremB}{Theorem B} 
\newtheorem*{theoremC}{Theorem C}

\DeclareMathOperator{\mre}{Re}

\usepackage{xcolor}

\begin{document} 
\title[Correction to: Weak product spaces of Dirichlet series]{Correction to: \\ Weak product spaces of Dirichlet series} 
\date{\today} 

\author{Ole Fredrik Brevig} 
\address{Department of Mathematics, University of Oslo, 0851 Oslo, Norway} 
\email{obrevig@math.uio.no}

\author{Karl-Mikael Perfekt} 	
\address{Department of Mathematical Sciences, Norwegian University of Science and Technology (NTNU), NO-7491 Trondheim, Norway} 
\email{karl-mikael.perfekt@ntnu.no}
\begin{abstract}
	We correct the proof of Theorem~8 in [\emph{Weak product spaces of Dirichlet series}, Integral Equations Operator Theory \textbf{86} (2016), no.~4, 453--473].
\end{abstract}


\maketitle

\section{Introduction} 
The claim ``Since $1/(1+\tau^2)$ is a Muckenhoupt $A_q$-weight for all $q>1$\ldots'' on page 465 of \cite{BP16} is false. Recall (from e.g.~\cite[Section~VI.6]{Garnett07}) that $w$ is called a \emph{Muckenhoupt} $A_q$-weight whenever
\[\sup_I \left(\frac{1}{|I|} \int_I w \right)\left(\frac{1}{|I|} \int_I w^{-\frac{1}{q-1}} \right)^{q-1} < \infty,\]
where the supremum is taken over all finite intervals $I$. For $w(\tau) = 1/(1+\tau^2)$ and $I = [-T,T]$, we deduce from the estimate $1+\tau^2 \geq \tau^2$ that 
\[\left(\frac{1}{|I|} \int_I w \right)\left(\frac{1}{|I|} \int_I w^{-\frac{1}{q-1}} \right)^{q-1} \geq \left(\frac{q-1}{q+1}\right)^{q-1} T \arctan{T}.\]
Consequently, the use of the non-tangential maximal function in \cite{BP16} is incorrect, and Theorem~8 of \cite{BP16} is unsubstantiated. In this corrigendum we remedy the issue in the context of Hardy spaces of Dirichlet series, and thereby correct the proof of Theorem~8. These results will be presented as Theorem~A and B below. The same mistake as in \cite{BP16} is repeated in the lead-up to the proof of \cite[Theorem 5.1]{BPS19}, and this is also rectified by Theorem~A and Theorem~B.

Unless otherwise stated, we retain the notation from \cite{BP16}. Specifically, we set $\Gamma_\tau = \{s=\sigma+it \,:\, |t-\tau|<\sigma\}$ for $\tau$ in $\mathbb{R}$ and consider the maximal function 
\begin{equation}\label{eq:maxf} 
	F^\ast(\tau) = \sup_{s \in \Gamma_\tau} |F(s)|. 
\end{equation}
The following result yields a maximal function characterization of $\mathscr{H}^p$, the converse inequality being a trivial consequence of \cite[Lemma~5]{Bayart02}. 
\begin{theoremA}
	Fix $0<p<\infty$. There is a constant $C_p \geq 1 $ such that
	\[\int_{\mathbb{T}^\infty} \int_{\mathbb{R}} |f_\chi^\ast(\tau)|^p \frac{d\tau}{\pi(1+\tau^2)} \leq C_p \|f\|_{\mathscr{H}^p}^p\]
	for every $f$ in $\mathscr{H}^p$. 
\end{theoremA}

Consider next the square function 
\begin{equation}\label{eq:sqf} 
	Sf(\tau) = \left(\int_{\Gamma_\tau} |f'(s)|^2\,ds \right)^\frac{1}{2}.
\end{equation}
The following square function characterization of $\mathscr{H}^p$ is precisely \cite[Theorem 8]{BP16}.
\begin{theoremB}
	Fix $0<p<\infty$. There are constants $A_p>0$ and $B_p>0$ such that
	\[A_p \|f\|_{\mathscr{H}^p}^p \leq \int_{\mathbb{T}^\infty } \int_{\mathbb{R}} (S f_\chi(\tau))^p\,\frac{d\tau}{\pi(1+\tau^2)}\,dm_\infty(\chi) \leq B_p \|f\|_{\mathscr{H}^p}^p\]
	for every $f$ in $\mathscr{H}^p$ with $f(+\infty)=0$. 
\end{theoremB}

We shall establish Theorem~A and Theorem~B by replacing the conformally invariant Hardy spaces $H^p_{\operatorname{i}}(\mathbb{C}_0)$ considered in \cite{BP16} with the (unweighted) Hardy spaces $H^p(\mathbb{C}_0)$. For a Dirichlet series $f$ in $\mathscr{H}^p$ and a character $\chi$ in $\mathbb{T}^\infty$, consider 
\begin{equation}\label{eq:Fchi} 
	F_\chi(s) = \frac{f_\chi(s)}{\pi^{1/p}(s+1)^{2/p}}. 
\end{equation}
It is well-known that if $f$ is in $\mathscr{H}^p$, then for almost every character $\chi$ in $\mathbb{T}^\infty$ the Dirichlet series $f_\chi$ is convergent in $\mathbb{C}_0$ and $F_\chi$ is in $H^p(\mathbb{C}_0)$. Moreover, $\|f\|_{\mathscr{H}^p}^p$ can be recovered by averaging $\|F_\chi\|_{H^p(\mathbb{C}_0)}^p$ over $\mathbb{T}^\infty$. See Lemma~\ref{lem:pnorm} below. Our arguments will be based on this, the maximal function characterization of $H^p(\mathbb{C}_0)$ due to Burkholder, Gundy and Silverstein \cite{BGS71}, and the square function characterization of $H^p(\mathbb{C}_0)$ due to Fefferman and Stein \cite{FS72}.

It is indeed straightforward to establish Theorem~A by comparing the maximal functions $f_\chi^\ast$ and $F_\chi^\ast$. However, for Theorem~B, the presence of a derivative in \eqref{eq:sqf} means that we need to compare the $L^2(\Gamma_\tau)$ integrals of $f_\chi'$ and 
\begin{equation}\label{eq:Fchidiff} 
	F_\chi'(s) = \frac{f_\chi'(s)}{\pi^{1/p}(s+1)^{2/p}} - \frac{2}{p} \frac{f_\chi(s)}{\pi^{1/p}(s+1)^{2/p+1}}. 
\end{equation}
One clearly expects the first term to be dominant, so our job is to prove that the second term can be disregarded. This will be achieved through a compactness argument which proceeds along standard lines. 

Certain technical issues arise since we are working with Hardy spaces of Dirichlet series. We rely mostly on the standard toolbox of $\mathscr{H}^p$ developed in \cite{Bayart02}, but one of our preliminary results does not appear in the literature and may be of some independent interest. To state it, we introduce the notation $T_\sigma f(s) = f(s+\sigma)$. 
\begin{theoremC}
	Fix $0<p<\infty$. If $f(s) = \sum_{n\geq N} a_n n^{-s}$ for some $N\geq1$, then
	\[\|T_\sigma f\|_{\mathscr{H}^p} \leq N^{-\sigma} \|f\|_{\mathscr{H}^p}\]
	for every $\sigma\geq0$. 
\end{theoremC}

Theorem~C is a direct consequence of the analogue of Hardy's convexity theorem for $\mathscr{H}^p$, which was recently established by the present authors in \cite{BP21} by adapting Hardy's original proof \cite{Hardy15}.

\subsection*{Organization} The present corrigendum is comprised of three additional sections. Section~\ref{sec:prelim} contains some preliminary results and definitions, including the proof of Theorem~C. Section~\ref{sec:A} and Section~\ref{sec:B} are devoted to the proofs of Theorem~A and Theorem~B, respectively. 

\subsection*{Acknowledgements} We thank Jiale Chen for making us aware of the error in \cite{BP16}.

\section{Preliminaries} \label{sec:prelim} 
For $0<p<\infty$, every Dirichlet series in $\mathscr{H}^p$ is absolutely convergent for $\mre{s}>1/2$. As elucidated in \cite[Section 2.1]{Bayart02}, it therefore follows from Birkhoff's ergodic theorem that the $\mathscr{H}^p$-norm of $T_\sigma f$ can computed using Carlson's formula whenever $\sigma>1/2$, 
\begin{equation}\label{eq:carlson} 
	\|T_\sigma f\|_{\mathscr{H}^p}^p = \lim_{T\to\infty} \frac{1}{2T} \int_{-T}^T |f(\sigma+it)|^p \,dt. 
\end{equation}
We are now ready to establish Theorem~C.
\begin{proof}
	[Proof of Theorem~C] We assume that $a_N\neq0$ because if not, we would get a stronger inequality by replacing $N$ by $N+1$. If $\sigma=0$ there is nothing to do. For $\sigma>0$ we will use Hardy's convexity theorem for $\mathscr{H}^p$ \cite[Theorem 3.8 (ii)]{BP21}, which states that the function $\sigma \mapsto \|T_\sigma f\|_{\mathscr{H}^p}$ is logarithmically convex on $[0,\infty)$. Applying this to the points $0$, $\sigma$ and $\kappa>\sigma$ we find that
	\[\|T_\sigma f \|_{\mathscr{H}^p} \leq \|f\|_{\mathscr{H}^p}^{1-\sigma/\kappa} \|T_\kappa f \|_{\mathscr{H}^p}^{\sigma/\kappa}.\]
	Our plan is now to estimate the final quantity as $\kappa \to \infty$. If we assume that $\kappa>1/2$, then we can use \eqref{eq:carlson} to compute $\|T_\kappa f \|_{\mathscr{H}^p}^p$. The absolute convergence and the assumption that $a_N \neq 0$ also imply that there is some $\kappa_0>1/2$ such that if $\kappa \geq \kappa_0$, then
	\[|f(\kappa+it)| \leq \sum_{n=N}^\infty \frac{|a_n|}{n^{\kappa}} \leq 2 |a_N| N^{-\kappa}.\]
	Combining these observations, we find that $\|T_\sigma f \|_{\mathscr{H}^p} \leq \|f\|_{\mathscr{H}^p}^{1-\sigma/\kappa} N^{-\sigma} (2|a_N|)^{\sigma/\kappa}$ for every $\kappa\geq\kappa_0$. Letting $\kappa\to \infty$ we obtain the stated result. 
\end{proof}

For $0<p<\infty$, let $H^p(\mathbb{C}_0)$ be the Hardy space of analytic functions in the right half-plane $\mathbb{C}_0 = \{s=\sigma+it\,:\, \sigma>0\}$, endowed with the (quasi-)norm 
\begin{equation}\label{eq:Hp} 
	\|F\|_{H^p(\mathbb{C}_0)}^p = \sup_{\sigma>0} \int_{-\infty}^\infty |F(\sigma+it)|^p\,dt. 
\end{equation}
It follows from Hardy's convexity theorem \cite{Hardy15} that the $\sup$ in \eqref{eq:Hp} can be replaced by the limit as $\sigma \to 0^+$. In combination with results extracted from \cite[Section~2.2]{Bayart02}, we have the following.
\begin{lemma}\label{lem:pnorm} 
	Fix $0<p<\infty$ and suppose that $f$ is in $\mathscr{H}^p$. For almost every $\chi$ in $\mathbb{T}^\infty$, the Dirichlet series $f_\chi$ converges in $\mathbb{C}_0$ and $F_\chi$ is in $H^p(\mathbb{C}_0)$. Moreover,
	\[\|f\|_{\mathscr{H}^p}^p = \int_{\mathbb{T}^\infty} \|F_\chi\|_{H^p(\mathbb{C}_0)}^p \,dm_\infty(\chi).\]
\end{lemma}

We have one more trivial preliminary result, which is stated separately for ease of reference, without proof.
\begin{lemma}\label{lem:trivial} 
	Decompose $\Gamma_\tau = \Gamma_{\tau,0} \cup \Gamma_{\tau,1}$, where
	\[\Gamma_{\tau,0} = \{s = \sigma+it \,:\, 0 < \sigma < 1 \text{ and } |t-\tau|<\sigma\}\]
	and $\Gamma_{\tau,1} = \Gamma_\tau \setminus \Gamma_{\tau,0}$. 
	\begin{enumerate}
		\item[(a)] If $s$ is in $\Gamma_\tau$, then $|s+1|^2 \geq (1+\tau^2)/2$. 
		\item[(b)] If $s$ is in $\Gamma_{\tau,0}$, then $|s+1|^2 \leq 6(1+\tau^2)$. 
	\end{enumerate}
\end{lemma}

\section{Proof of Theorem A} \label{sec:A} 
We know from \cite[Theorem~1]{BGS71} that there is a constant $c_p\geq 1$ such that 
\begin{equation}\label{eq:MFLB} 
	\int_{\mathbb{R}} |F^\ast(\tau)|^p\,d\tau \leq c_p \|F\|_{H^p(\mathbb{C}_0)}^p 
\end{equation}
for every $F$ in $H^p(\mathbb{C}_0)$. The plan is to apply \eqref{eq:MFLB} to the functions $F_\chi$ from \eqref{eq:Fchi} and then to use Lemma~\ref{lem:pnorm}. Decompose $\Gamma = \Gamma_{\tau,0}\cup \Gamma_{\tau,1}$ as above. The corresponding maximal functions will be denoted $F_0^\ast$ and $F_1^\ast$, so that
\[F^\ast(\tau) = \max\big(F_0^\ast(\tau),F_1^\ast(\tau)\big).\]
It follows from Lemma~\ref{lem:trivial} (b) that 
\begin{equation}\label{eq:chi0} 
	|F_{\chi,0}^\ast(\tau)|^p \geq \frac{|f_{\chi,0}^\ast(\tau)|^p}{6\pi (1+\tau^2)}. 
\end{equation}
Using Lemma~\ref{lem:trivial} (a) and the pointwise estimate $|f(s)|^p \leq \zeta(2\mre{s})\|f\|_{\mathscr{H}^p}^p$, valid for $\mre{s}>1/2$, from \cite[Theorem 3]{Bayart02}, we find that 
\begin{equation}\label{eq:chi1} 
	|F_{\chi,1}^\ast(\tau)|^p \leq \frac{2\zeta(2) \|f\|_{\mathscr{H}^p}^p}{\pi(1+\tau^2)}. 
\end{equation}
Exploiting the inequalities $\max(a^p,b^p) \geq a^p-b^p$ and $a^p+b^p \geq \max(a^p,b^p)$ for $a,b\geq0$, we obtain from \eqref{eq:chi0}, \eqref{eq:chi1}, and the pointwise estimate again, that
\[|F_\chi^\ast(\tau)|^p \geq \frac{|f_{\chi,0}^\ast(\tau)|^p}{6 \pi (1+\tau^2)} - \frac{2\zeta(2) \|f\|_{\mathscr{H}^p}^p}{\pi(1+\tau^2)} \geq \frac{|f_{\chi}^\ast(\tau)|^p}{\pi (1+\tau^2)} - \frac{13\zeta(2) \|f\|_{\mathscr{H}^p}^p}{6\pi(1+\tau^2)}.\]
Inserting this into \eqref{eq:MFLB}, integrating over $\chi$ and using Lemma~\ref{lem:pnorm}, we find that
\[c_p \|f\|_{\mathscr{H}^p}^p \geq \int_{\mathbb{T}^\infty} \int_{\mathbb{R}} |f_{\chi}^\ast(\tau)|^p\,\frac{d\tau}{\pi (1+\tau^2)}\, dm_\infty(\chi) - \frac{13\pi^2}{36} \|f\|_{\mathscr{H}^p}^p.\]
This implies the stated result with $C_p = c_p+13\pi^2/36$. \qed

\section{Proof of Theorem B} \label{sec:B} 
Throughout this section, we will use the triangle inequality in its weak form,
\[|a+b|^p \leq 2^p (|a|^p + |b|^p),\]
which holds for $0<p<\infty$.
\begin{proof}
	[Proof of the upper bound in Theorem~B] By the square function characterization of $H^p(\mathbb{C}_0)$ from \cite[Theorem~8]{FS72} and Lemma~\ref{lem:pnorm}, there is a constant $b_p>0$ such that
	\[b_p \|f\|_{\mathscr{H}^p}^p \geq \int_{\mathbb{T}^\infty} \int_{\mathbb{R}} (SF_\chi(\tau))^p \,d\tau\,dm_\chi(\infty).\]
	We get a lower bound for the right hand side by restricting the integral over $\Gamma_\tau$ in the definition of the square function \eqref{eq:sqf} to $\Gamma_{\tau,0}$ and using the estimate from Lemma~\ref{lem:trivial} (b) therein to conclude that
	\[b_p \|f\|_{\mathscr{H}^p}^p \geq \frac{1}{6}\int_{\mathbb{T}^\infty} \int_{\mathbb{R}} \left(\int_{\Gamma_{\tau,0}} \left|f'_\chi(s)-\frac{2}{p}\frac{f_\chi(s)}{(s+1)}\right|^2 \,ds \right)^{\frac{p}{2}} \,\frac{d\tau}{\pi(1+\tau^2)}\,dm_\infty(\chi).\]
	From the triangle inequality for $L^2(\Gamma_\tau)$ and the triangle inequality, we find that 
	\begin{multline*}
		b_p \|f\|_{\mathscr{H}^p}^p \geq \frac{2^{-p}}{6} \int_{\mathbb{T}^\infty} \int_{\mathbb{R}} \left(\int_{\Gamma_{\tau,0}} |f'_\chi(s)|^2 \,ds \right)^{\frac{p}{2}} \,\frac{d\tau}{\pi(1+\tau^2)}\,dm_\infty(\chi) \\
		- \frac{1}{6}\int_{\mathbb{T}^\infty} \int_{\mathbb{R}} \left(\int_{\Gamma_{\tau,0}} \left|\frac{2}{p}\frac{f_\chi(s)}{(s+1)}\right|^2 \,ds \right)^{\frac{p}{2}} \,\frac{d\tau}{\pi(1+\tau^2)}\,dm_\infty(\chi). 
	\end{multline*}
	Using Theorem~A and the estimate $|s+1|\geq1$, valid for $s$ in $\mathbb{C}_0$, we conclude that
	\[\frac{1}{6}\int_{\mathbb{T}^\infty} \int_{\mathbb{R}} \left(\int_{\Gamma_{\tau,0}} \left|\frac{2}{p}\frac{f_\chi(s)}{(s+1)}\right|^2 \,ds \right)^{\frac{p}{2}} \,\frac{d\tau}{\pi(1+\tau^2)}\,dm_\infty(\chi) \leq \left(\frac{2}{p}\right)^p \frac{C_p}{6} \|f\|_{\mathscr{H}^p}^p.\]
	Letting $\widetilde{B}_p = 6 \cdot 2^p \cdot b_p + (4/p)^p \cdot C_p$, we thus have
	\[\widetilde{B}_p \|f\|_{\mathscr{H}^p}^p \geq \int_{\mathbb{T}^\infty} \int_{\mathbb{R}} \left(\int_{\Gamma_{\tau,0}} |f'_\chi(s)|^2 \,ds \right)^{\frac{p}{2}} \,\frac{d\tau}{\pi(1+\tau^2)}\,dm_\infty(\chi).\]
	Another application of the triangle inequality now gives that 
	\begin{multline*}
		\widetilde{B}_p \|f\|_{\mathscr{H}^p}^p \geq 2^{-\frac{p}{2}} \int_{\mathbb{T}^\infty} \int_{\mathbb{R}} (Sf_\chi(\tau))^p\,\frac{d\tau}{\pi(1+\tau^2)}\,dm_\infty(\chi) \\
		- \int_{\mathbb{T}^\infty} \int_{\mathbb{R}} \left(\int_{\Gamma_{\tau,1}} |f'_\chi(s)|^2 \,ds \right)^{\frac{p}{2}} \,\frac{d\tau}{\pi(1+\tau^2)}\,dm_\infty(\chi). 
	\end{multline*}
	The final term can be estimated in different ways. One possibility is to apply Cauchy's formula to the function $g_\chi = T_{\sigma-1}(f_\chi-f(+\infty))$, where $s = \sigma + it$, integrating over the circle with centre $1+it$ and radius $1/3$, to conclude that
	\[|f_\chi'(s)| \leq 3 \zeta(4/3)^{1/p} \|T_{\sigma-1}(f-f(+\infty))\|_{\mathscr{H}^p} \leq 3 \zeta(4/3)^{1/p} 2^{-\sigma+1+1/p} \|f\|_{\mathscr{H}^p},\]
	where we also used the pointwise estimate of \cite[Theorem 3]{Bayart02}, Theorem~C, the triangle inequality, and the estimate $|f(+\infty)|\leq\|f\|_{\mathscr{H}^p}$. Applying this estimate and extending the innermost integral to $\Gamma_\tau$, we obtain by computation the stated result with $B_p = 2^{p/2}\widetilde{B}_p + 2\zeta(4/3) (6/\log{2})^p$. 
\end{proof}

As mentioned in the introduction, the main difficulty in the proof of Theorem~B is to control the contribution from the second term in \eqref{eq:Fchidiff}. This will be achieved through the following result.
\begin{lemma}\label{lem:Rpf} 
	Fix $0<p<\infty$ and define
	\[\mathscr{R}_p(f) = \int_{\mathbb{T}^\infty} \int_{\mathbb{R}} \left(\int_{\Gamma_\tau} |f_\chi(s)|^2\,ds\right)^{\frac{p}{2}}\,\frac{d\tau}{\pi(1+\tau^2)} \,dm_\infty(\chi).\]
	There is a constant $D_p>0$ such that if $f(s) = \sum_{n\geq N} a_n n^{-s}$ for some $N\geq2$, then
	\[\mathscr{R}_p(f) \leq \frac{D_p}{(\log{N})^{p}} \|f\|_{\mathscr{H}^p}^p.\]
\end{lemma}
\begin{proof}
	It is instructive to first consider the case $p=2$. By interchanging the order of integration and using Theorem~C, we find that 
	\begin{align*}
		\mathscr{R}_2(f) &\leq \int_{\mathbb{T}^\infty} \int_{\mathbb{R}} \int_{\Gamma_\tau} |f_\chi(s)|^2\,ds\,\frac{d\tau}{\pi(1+\tau^2)}\,dm_\infty(\chi) \\
		&= \int_{\mathbb{R}}\int_{\Gamma_\tau} \|T_\sigma f\|_{\mathscr{H}^2}^2 \,ds \,\frac{d\tau}{\pi(1+\tau^2)} \\
		&\leq \|f\|_{\mathscr{H}^2} \int_{\mathbb{R}}\int_{\Gamma_\tau} N^{-2\sigma} \,ds \,\frac{d\tau}{\pi(1+\tau^2)} = \frac{1}{2} \frac{\|f\|_{\mathscr{H}^2}^2}{(\log{N})^2}. 
	\end{align*}
	This is the stated result with $D_2=1/2$. In the case $0<p<2$, we first extract a maximal function to obtain
	\[\mathscr{R}_p(f) \leq \int_{\mathbb{T}^\infty} \int_{\mathbb{R}} (f_\chi^\ast(\tau))^{\frac{p(2-p)}{2}} \left(\int_{\Gamma_\tau} |f_\chi(s)|^p\,ds\right)^{\frac{p}{2}}\,\frac{d\tau}{\pi(1+\tau^2)}\,dm_\infty(\chi).\]
	Next we use H\"older's inequality and Theorem~A, which yields that
	\[\mathscr{R}_p(f) \leq C_p^{\frac{2}{2-p}} \|f\|_{\mathscr{H}^p}^{\frac{(2-p)p}{2}} \left(\int_{\mathbb{T}^\infty} \int_{\mathbb{R}}\int_{\Gamma_\tau} |f_\chi(s)|^p\,ds\,\frac{d\tau}{\pi(1+\tau^2)}\,dm_\infty(\chi)\right)^{\frac{p}{2}}.\]
	Using Theorem~C as before, we get the stated result with $D_p = C_p^{2/(2-p)} 2^{p/2} p^{-p}$. In the case $2<p<\infty$, we first use H\"older's inequality to estimate
	\[\int_{\Gamma_\tau} |f_\chi(s)|^2 \,ds \leq \int_0^\infty (2\sigma)^{\frac{p-2}{p}}\left(\int_{\tau-\sigma}^{\tau+\sigma} |f_\chi(\sigma+it)|^p \,dt\right)^\frac{2}{p}\,d\sigma.\]
	Using this estimate and Minkowski's inequality, we find that $\mathscr{R}_p(f)$ is bounded from above by
	\[ \left(\int_0^\infty\left(\int_{\mathbb{T}^\infty} \int_{\mathbb{R}} (2\sigma)^{\frac{p-2}{2}} \int_{\tau-\sigma}^{\tau+\sigma} |f_\chi(\sigma+it)|^p \,dt \,\frac{d\tau}{\pi(1+\tau^2)}\,dm_\infty(\chi) \right)^\frac{2}{p} \,d\sigma\right)^\frac{p}{2}.\]
	Interchanging the order of integration and using Theorem~C, we find that
	\[\int_{\mathbb{T}^\infty} \int_{\mathbb{R}} (2\sigma)^{\frac{p-2}{2}} \int_{\tau-\sigma}^{\tau+\sigma} |f_\chi(\sigma+it)|^p dt \frac{d\tau}{\pi(1+\tau^2)}dm_\infty(\chi) \leq (2\sigma)^{p/2} N^{-\sigma p} \|f\|_{\mathscr{H}^p}^p.\]
	Computing the final integral over $\sigma$ yields the stated result with $D_p = 2^{-p/2}$. 
\end{proof}

For $N\geq2$, we let $S_N$ denote the partial sum operator
\[S_N\left(\sum_{n=1}^\infty a_n n^{-s}\right) = \sum_{n< N} a_n n^{-s}.\]
Lemma~\ref{lem:Rpf} will be used to establish the following.
\begin{lemma}\label{lem:contradictme} 
	Fix $0<p<\infty$. There are constants $E_p>0$ and $N_p \geq 2$ such that
	\[E_p \|f\|_{\mathscr{H}^p}^p \leq \int_{\mathbb{T}^\infty} \int_{\mathbb{R}} (Sf_\chi(\tau))^p \,\frac{d\tau}{\pi(1+\tau^2)}\,dm_\infty(\chi) + \|S_{N_p} f\|_{\mathscr{H}^p}^p\]
	for every $f$ in $\mathscr{H}^p$ with $f(+\infty)=0$.
\end{lemma}
\begin{proof}
	Let $N\geq2$ be a number to be chosen later. By the square function characterization of $H^p(\mathbb{C}_0)$ from \cite[Theorem~8]{FS72} and Lemma~\ref{lem:pnorm}, there is a constant $a_p>0$ such that
	\[a_p \|f\|_{\mathscr{H}^p}^p \leq \int_{\mathbb{T}^\infty} \int_{\mathbb{R}} (SF_\chi(\tau))^p \,d\tau \,dm_\infty(\chi).\]
	Recalling \eqref{eq:Fchidiff}, we use the triangle inequality for $L^2(\Gamma_\tau)$, the triangle inequality, Lemma~\ref{lem:trivial} (a), the estimate $|s+1|^2 \geq 1$, and the estimate $(2/p)^p \leq e^{2/e}$, to see that
	\[\frac{a_p}{2^{p+1} e^{2/e}} \|f\|_{\mathscr{H}^p}^p \leq \int_{\mathbb{T}^\infty} \int_{\mathbb{R}} (Sf_\chi(\tau))^p \,\frac{d\tau}{\pi(1+\tau^2)}\,dm_\infty(\chi) + \mathscr{R}_p(f),\]
	where $\mathscr{R}_p(f)$ is as in Lemma~\ref{lem:Rpf}. To actually apply Lemma~\ref{lem:Rpf}, we need to subtract $S_N f$ from $f$.
	Using the triangle inequality, Lemma~\ref{lem:Rpf} twice, and then the triangle inequality again with the trivial estimates $1 \leq 2^p$ and $N\geq2$, we find that
	\[\mathscr{R}_p(f) \leq 4^p D_p \left(\frac{2\|S_N f\|_{\mathscr{H}^p}^p}{(\log{2})^p} + \frac{\|f\|_{\mathscr{H}^p}^p}{(\log{N})^p}\right)\]
	We obtain the stated result with $E_p = (\log{2})^{p} a_p/(8^{p+1} e^{2/e} D_p)$, if $N = N_p\geq 2$ is chosen so large that
	\[\frac{4^p D_p}{(\log{N_p})^p} \leq \frac{a_p}{2^{p+2} e^{2/e}}. \qedhere\]
\end{proof}
\begin{proof}
	[Proof of the lower bound in Theorem~B] We will argue by contradiction. Assume that there is no constant $A_p>0$ such that
	\[A_p \|f\|_{\mathscr{H}^p}^p \leq \int_{\mathbb{T}^\infty } \int_{\mathbb{R}} (S f_\chi(\tau))^p\,\frac{d\tau}{\pi(1+\tau^2)}\,dm_\infty(\chi)\]
	for all $f$ in $\mathscr{H}^p$ with $f(+\infty)=0$. We can then find a sequence of functions $(f_k)_{k\geq1}$ in $\mathscr{H}^p$ such that $\|f_k\|_{\mathscr{H}^p} = 1$ for every $k\geq1$, but such that 
	\begin{equation}\label{eq:squarelim} 
		\lim_{k\to \infty } \int_{\mathbb{T}^\infty } \int_{\mathbb{R}} (S f_{k,\chi}(\tau))^p\,\frac{d\tau}{\pi(1+\tau^2)}\,dm_\infty(\chi) = 0. 
	\end{equation}
	If we can derive from \eqref{eq:squarelim} that also $\|S_{N_p} f_k\|_{\mathscr{H}^p} \to 0$ as $k\to\infty$, we will obtain a contradiction to Lemma~\ref{lem:contradictme}. In view of the trivial estimate
	\begin{equation} \label{eq:SNpfk}
		\|S_{N_p} f_k\|_{\mathscr{H}^p} \leq \|S_{N_p} f_k\|_{\mathscr{H}^\infty} \leq \sum_{n<N_p} |a_{n,k}|,
	\end{equation}
	it is enough to deduce from \eqref{eq:squarelim} that $a_{n,k} \to 0$ as $k\to\infty$ for every $n < N_p$.
	
	Let us first consider the simpler case $1 \leq p < \infty$. Since $\|f_k\|_{\mathscr{H}^p}=1$ we know that $|a_{n,k}|\leq 1$ for all $n\geq2$. By passing to subsequences, we may assume that there are complex numbers $a_n$ such that $a_{n,k} \to a_n$ for $n<N_p$. If $a_2\neq0$, we can find $K_2$ such that if $k \geq K_2$, then $|a_{2,k}| \geq |a_2|/2$. In this case, pick $\kappa_2$ so large that if $\sigma\geq \kappa_2$, then 
	\[|f_{k,\chi}'(s)| \geq \frac{|a_{2,k}| (\log{2})}{2^{\sigma}} - \sum_{n > 2} (\log{n}) n^{-\sigma} \geq \frac{|a_2| (\log{2})}{2^{\sigma+1}} - \sum_{n > 2} (\log{n}) n^{-\sigma} \geq \frac{|a_2| (\log{2})}{2^{\sigma+2}}.\]
	Restricting the integral over $\Gamma_\tau$ in \eqref{eq:squarelim} to $\mre{s}\geq \kappa_2$, we insert this estimate and let $k \to \infty$ to obtain a contradiction to the assumption that $a_2 \neq 0$. Hence $a_2=0$. Similarly, if $a_3\neq0$, we may find $K_3$ and $\kappa_3$ such that if $k \geq K_3$ and $\mre{s} \geq \kappa_3$, then
	\[|f_{k,\chi}'(s)| \geq |a_3| (\log{3}) 3^{-\sigma}/4 - |a_{2,k}| (\log{2}) 2^{-\sigma}.\]
	We will now restrict our attention to the strip $\kappa_3 \leq \mre{s} \leq \kappa_3+1$. Since $a_{2,k}\to 0$, we can choose $K_{2,3}$ so large that if $k\geq \max(K_{2,3},K_3)$ and $\kappa_3 \leq \sigma \leq \kappa_3+1$, then
	\[|a_{2,k}| (\log{2}) 2^{-\sigma} \leq |a_3| (\log{3}) 3^{-\sigma}/8.\]
	This implies that $|f_{k,\chi}'(s)| \geq |a_3| (\log{3}) 3^{-\sigma}/8$ in the strip $\kappa_3 \leq \sigma \leq \kappa_3+1$ for every sufficiently large $k$. Restricting the integral over $\Gamma_\tau$ in \eqref{eq:squarelim} to this strip and letting $k \to \infty$, we find that $a_3=0$. We continue inductively and conclude that $a_n=0$ for every $n<N_p$ and obtain the desired contradiction to Lemma~\ref{lem:contradictme} through \eqref{eq:SNpfk}.
	
	The same argument works in the case $0<p<1$ with one minor modification. We no longer have the estimate $|a_{n,k}|\leq 1$, but by \cite[Theorem~9]{Bayart02} it holds that
	\[|a_{n,k}| \leq (2/p)^{\Omega(n)/2} \|f_k\|_{\mathscr{H}^p}^p = (2/p)^{\Omega(n)/2},\]
	where $\Omega(n)$ denotes the number of prime factors of $n$ counting multiplicities. By the estimate $\Omega(n) \leq \log{n}/\log{2}$ and a straightforward rough calculus estimate, we find that $|a_{n,k}| \leq n^{1/p}$. As a consequence, we simply need to choose a slightly larger $\kappa_n$ for each $n<N_p$, and otherwise argue in a more or less identical manner.
\end{proof}

\bibliographystyle{amsplain-nodash} 
\bibliography{weakcorr} 

\end{document}